\pdfoutput=1

\documentclass[11pt]{amsart}
\usepackage{amsmath}
\usepackage{amsthm}
\usepackage{amsfonts}
\usepackage{amssymb}
\usepackage{appendix}

\usepackage[parfill]{parskip}    

\usepackage{geometry}                
\usepackage{graphicx}
\usepackage{epstopdf}

\usepackage[all,cmtip]{xy}

\newtheorem{theorem}{Theorem}
\newtheorem{definition}[theorem]{Definition}
\newtheorem{lemma}[theorem]{Lemma}
\newtheorem{corollary}[theorem]{Corollary}
\newtheorem{proposition}[theorem]{Proposition}

\newcommand{\x}{\mathbf x}

\newcommand{\xix}{\xi}
\newcommand{\xiy}{\eta}
\newcommand{\xiz}{\zeta}

\newcommand{\ba}{\begin{align} }
\newcommand{\bs}{\begin{split} }
\newcommand{\es}{\end{split} }
\newcommand{\ea}{\end{align} }

\newcommand{\UU}[1]{\mathcal{U}^{(#1)} }
\newcommand{\PU}[1]{\mathcal{R}^{(#1)} }
\newcommand{\OUU}[1]{\overline{\mathcal{U}}^{(#1)} }
\newcommand{\XU}[2]{\mathcal{X}^{(#1)}_{#2}}

\newcommand{\pii}{i} 
\newcommand{\pij}{j} 

\newcommand{\HH}{\mathbf H}
\newcommand{\cH}{\mathcal{H}}
\newcommand{\cV}{\mathcal{V}}

\newcommand{\cS}{\mathcal{S}}
\newcommand{\cT}{\mathcal{T}}
\newcommand{\PP}[1]{{P}^{#1}}
\newcommand{\tPP}[1]{\tilde{{P}}^{#1}}
\newcommand{\RR}{\mathbb R}
\newcommand{\rp}{ {\hat{K} } }
\newcommand{\ip}{  {K_{\infty} }  }
\providecommand{\abs}[1]{\left\lvert#1\right\rvert}
\providecommand{\norm}[1]{\lVert#1\rVert}

\providecommand{\col}[3]{\begin{pmatrix} #1 \\ #2 \\ #3 \end{pmatrix} }

\providecommand{\pd}[2]{\frac{\partial #1}{\partial #2}}

\providecommand{\set}[1]{\left\{ #1 \right \}}
\providecommand{\inner}[1]{\left( #1 \right )}
\newcommand{\inc}[1]{\Upsilon_{#1}}
\newcommand{\prox}[1]{\mathcal{I}_{#1}}
\newcommand{\QS}{S}

\providecommand{\w}[1]{w^{(#1)}}

\DeclareMathOperator{\curl}{{curl}}
\DeclareMathOperator{\spnn}{{span}}
\DeclareMathOperator{\grad}{{grad}}
\DeclareMathOperator{\divv}{div}

\DeclareMathOperator{\trace}{tr}

\providecommand{\spn}[1]{\spnn \left\{ #1 \right\} }

\title{Numerical integration for high order pyramidal finite elements}
\author{Nilima Nigam}
\address{Department of Mathematics, Simon Fraser University, Burnaby, V5A 1S6}
\email{nigam@math.sfu.ca}
\author{Joel Phillips}
\address{Department of Mathematics and Statistics, McGill University, Montreal, H3A 2K6}
\email{phillips@math.mcgill.ca}
\thanks{The work of NN was supported by the Natural Sciences and Engineering Research Council of Canada, and the Canada Research Chairs program. JP was supported by a Natural Sciences and Engineering Research Council graduate fellowship.}
\date{\today}                                           

\begin{document}

\begin{abstract}
We examine the effect of numerical integration on the convergence of high order pyramidal finite element methods.  Rational functions are indispensable to the construction of pyramidal interpolants so the conventional treatment of numerical integration, which requires that the finite element approximation space is piecewise polynomial, cannot be applied.  We develop an analysis that allows the finite element approximation space to include rational functions and show that despite this complication, conventional rules of thumb can still  be used to select appropriate quadrature methods on pyramids.  Along the way, we present a new family of high order pyramidal finite elements for each of the spaces of the de Rham complex.
\end{abstract}
\maketitle


\section{Introduction}
Pyramidal finite elements are used in applications as ``glue'' in heterogeneous meshes containing hexahedra, tetrahedra and prisms.   Various constructions of high order pyramidal elements have been proposed \cite{zgainski:edge, zgainski:family, graglia:highorderpyramid, hiptmair:pyramid, zaglmayr:thesis, phillips:pyramid}.  A useful summary of the approaches taken for $H^1$-conforming elements is given by Bergot et al. \cite{bergotcohendurufle::pyramid}, who also provide some motivating numerical results for the performance of methods based on meshes containing pyramidal elements.  

Our aim here is to study the effect of numerical integration on arbitrarily high order pyramidal finite elements that approximate each of the spaces of the de Rham complex.  If they are to be used to implement stable mixed methods, such elements should also satisfy a commuting diagram property. One such set of elements was constructed by Zaglmayr based on the theory of local exact sequences, \cite{zaglmayr:thesis}, and is summarised in \cite{demkowicz2007computing}.  Another construction was given by the authors in \cite{phillips:pyramid}, and forms the starting point for this work.

A prototypical (linear) problem is 
\begin{align}\label{protobilinear}
\text{For $a:V \times V \rightarrow \RR$ and $f \in V'$, find $u \in V$ such that:}\quad a(u,v) = f(v) \quad \forall\: v \in V,
\end{align}
where $V$ is a space of functions on a domain, $\Omega \subset \RR^n$.  One way of obtaining a numerical approximation to $u$ is to replace $V$ with some finite dimensional approximation, $V_h$ constructed using finite elements on a mesh whose size is controlled by $h$.  A typical result is that the approximate solution converges to the true solution at some rate, $O(h^k)$ where the order of convergence, $k$, depends on the degree of largest complete space of polynomials used in the finite element approximation space.  

In general, the bilinear form, $a(\cdot,\cdot)$ and the right hand side $f(\cdot)$ are evaluated using numerical integration rules.  These are additional sources of errors in the approximate solution.  The theory that describes these errors is now classical and can be found, for example, in \cite{ciarlet:fem, brenner:mtfem}; its objectives, nicely summed up in \cite{ciarlet:fem}, are

\begin{quote}``to give sufficient conditions on the quadrature scheme which insure that the effect of the numerical integration does decrease [the] order of convergence.''
\end{quote}

In this paper we will show that the quadratures described as conical product formulae by Stroud \cite{stroud:acmiac} satisfy the above property for our pyramidal elements.  The main challenge arises from the fact that the classical theory is only applicable to finite elements with approximation spaces consisting purely of polynomials, but pyramidal elements necessarily include functions other than polynomials, specifically rational functions (see \cite{phillips:pyramid} or \cite{wieners:conforming}).  In contrast to the claim in \cite{bergotcohendurufle::pyramid}, we show that the importance of these functions in constructing interpolants means that it is not possible to achieve global estimates of the consistency error by summing element-wise estimates that only deal with polynomials.    

Section \ref{sectiondefinitions} introduces a framework that will allow us to unify our analysis for discrete approximations to each of the spaces of the de Rham complex.  We also recall the definitions of the approximation spaces for the elements in \cite{phillips:pyramid} and the quadrature rules given in \cite{stroud:acmiac}.  In section \ref{sectionstroud} we show that the conical product formulae are exact for products of all pairs of functions from the approximation spaces, including the non-polynomials.  The intuition from the classical theory would be that this is all that is required.  However, in section \ref{sectionexisting} we show that the reasoning behind this intuition is insufficient when functions other than polynomials are present.  To overcome this, we derive a generalisation of the standard Bramble-Hilbert argument.  In section \ref{newapproxspaces} we present new families of approximation spaces that allow us to take advantage of this generalisation (and which can be used to construct new pyramidal finite elements in their own right).  Finally, we pull everything together in section \ref{sectionpyramid} and show that Stroud's quadrature rules satisfy the desired property for both the new and original families of elements.

\section{Definitions}\label{sectiondefinitions}
\subsection{Differential forms}
In common with Arnold et al. \cite{arnold:acta} we find that it is natural to use tools from differential geometry when discussing approximations to the de Rham complex.  

Let $\Omega \subset \RR^n$ and define $\Lambda^{(s)}(\Omega)$ as the space of differential $s$-forms on $\Omega$.   A point, $x \in \Omega$, has coordinates $(x^i)_{i=0\dots n}$ and a given $u \in \Lambda^{(s)}(\Omega)$ can be expressed in terms of its components,  $u = \sum_\alpha u_\alpha dx^{\alpha_1} \wedge \cdots \wedge dx^{\alpha_s}$ where each $u_\alpha \in C^{\infty}(\Omega)$ and the multi-indices, $\alpha = \alpha_1\cdots\alpha_s$ run over the set, $\inc{s}$, of all increasing sequences, $\{1... s\} \rightarrow \{1...n\}$.

Define $\Theta^{(s)}(\Omega)$ to be the space of all (covariant) tensors, $A : \Lambda^{(s)}(\Omega) \times \Lambda^{(s)}(\Omega) \rightarrow C^\infty(\Omega)$ that can be defined in terms of the pointwise representation, 
\begin{align}\label{tensorrepresentation}
A(u,v)(x) := A^{\alpha\beta}(x) u_\alpha(x) v_\beta(x) \quad \forall u,v \in \Lambda^{(s)}(\Omega),
\end{align}
where we are using the Einstein summation convention, $A^{\alpha\beta} u_\alpha v_\beta:= \sum_{\alpha, \beta \in \inc{s}} A^{\alpha\beta} u_\alpha v_\beta$.  We will insist that $A^{\alpha\beta}$ is anti-symmetric in the first $s$ and second $s$ components, which makes the representation unique.

A tensor, $A \in \Theta^{(s)}(\Omega)$ induces a bilinear form on $\Lambda^{(s)}(\Omega)$:
\begin{align*}
\inner{u,v}_{A,\Omega} =\int_{\Omega} A^{\alpha\beta}(x) u_\alpha(x) v_\beta(x) d\x,  
\end{align*}
where $d\x =  dx^{\alpha_1} \wedge \cdots \wedge dx^{\alpha_n}$.

Let $\mathcal{T}$ be a partition of $\Omega$ where every $K \in \mathcal{T}$ is the image of a simple reference domain, $\rp \subset \RR^n$, under a diffeomorphism $\phi_K: \rp \rightarrow K$.  On each $K$, the \emph{reference coordinates}, $\hat x = (x^{\hat i})_{\hat i=0\dots n}$ of any point $x \in K$, are given by $\hat x = \phi_K^{-1}(x)$.  Given $u \in \Lambda^{(s)}(K)$, the reference coordinate system induces a new set of components $u_{\hat \alpha}$.  Differential forms are contravariant, so the components transform as:
\begin{align}\label{formtransform}
u_{\hat{\alpha}} = \sum_{\alpha \in \inc{s}} \pd{x^{\alpha_1}}{x^{\hat{\alpha}_1}} \cdots  \pd{x^{\alpha_s}}{{x}^{\hat{\alpha}_s}} u_\alpha.
\end{align}
The components of a covariant tensor, $A \in \Theta^{(s)}(\Omega)$ transform as:
\begin{align}\label{tensortransform}
A^{\hat{\alpha}\hat{\beta}} = \sum_{\alpha, \beta \in \inc{s}} \pd{{x}^{\hat{\alpha}_1}}{x^{\alpha_1}} \cdots  \pd{{x}^{\hat{\alpha}_s}}{x^{\alpha_s}}\pd{{x}^{\hat{\beta}_1}}{x^{\beta_1}} \cdots  \pd{{x}^{\hat{\beta}_s}}{x^{\beta_s}} A^{\alpha\beta}.
\end{align}
Note that $\langle u(x),v(x) \rangle_{A(x)} =  A^{\alpha \beta}(x) u_{\alpha}(x)v_\beta(x) = A^{\hat \alpha \hat \beta}(\hat x) u_{\hat \alpha}(\hat x) v_{\hat \beta}(\hat x)$ is just a 0-form and we have the change of variables formula on each element, $K$: 
\begin{align}\label{bformchangevariables}
(u,v)_{A,K} = \int_{K} A^{\alpha\beta} u_\alpha v_\beta d\x = \int_{\hat K} A^{\hat \alpha\hat \beta} u_{\hat \alpha} v_{\hat \beta} \abs{D \phi_K} d \hat \x,
\end{align}
where $D\phi_K$ is the Jacobian of $\phi_K$ and $d\hat \x = dx^{\hat \alpha_1} \wedge \cdots \wedge dx^{\hat \alpha_n}$.

When $n=2$ and $n=3$, it is conventional to think of differential forms in terms of proxy fields.  The spaces $\Lambda^{(0)}(\Omega)$ and $\Lambda^{(n)}(\Omega)$ are always isomorphic to the scalar field, $C^\infty(\Omega)$.  When $n=3$, the spaces  $\Lambda^{(1)}(\Omega)$ and $\Lambda^{(2)}(\Omega)$ are isomorphic to the vector field, $(C^\infty(\Omega))^3$.  For $u \in \Lambda^{(s)}(\Omega)$, we denote the components of the proxy field as $u_i$ for $i \in \prox{s} = \left\{1, \dots, \binom{3}{s} \right\}$.  The isomorphisms for the vector fields are given by
\begin{align*}
u \in \Lambda^{(1)}(\Omega) \mapsto\col{u_{1}}{u_{2}}{u_{3}},  \quad 
u \in \Lambda^{(2)}(\Omega) \mapsto\col{u_{23}}{-u_{13}}{u_{12}}. 
\end{align*}
With these identifications, the exterior derivatives, $d:\Lambda^{(s)}(\Omega) \rightarrow \Lambda^{(s+1)}(\Omega)$ for $s=0,1,2$ become the familiar $\grad$, $\curl$ and $\divv$. 

As with the differential forms, we will use symbols on the subscripts (and superscripts) of proxies to indicate the coordinate system that is being used to determine the components of the proxy fields.  Given some $u \in \Lambda^{(s)}(\Omega)$, $u_{i'}$ is the $i$th component of its proxy in the coordinate system $x' = \left(x^{1'}, x^{2'}, x^{3'}\right)$.  We will also write $u' = (u_{i'})_{i \in \prox{s}}$ to indicate all the components of the vector (or scalar) field.

For a coordinate change, $x = \phi(x')$, the weights appearing in the contravariant and covariant transformation rules, \eqref{formtransform} and \eqref{tensortransform}, can be written as the entries of a $\binom{3}{s} \times \binom{3}{s}$ matrix, $\w{s}_{\phi}$.  We choose to let $\w{s}_{\phi}$ to be the weight in the covariant transformation so that, for $u \in \Lambda^{(s)}(\Omega)$
\begin{align}\label{proxytransform}
\sum_{\pii' \in \prox{s}} \left(\w{s}_{\phi}\right)_{\pii, \pii'} u_{\pii'} = u_\pii \quad \forall\pii \in \prox{s}. 
\end{align}
The weights can be calculated in terms of the Jacobian, $D\phi$:
\begin{align}\label{pullbackweights}
\w{0}_\phi = 1,\qquad \w{1}_\phi = D{\phi ^{-1}}^t, \qquad \w{2}_\phi = \abs{D \phi ^{-1}}D \phi, \qquad \w{3}_\phi = \abs{D \phi ^{-1}}. 
\end{align}

The exterior derivative is an intrinsic property of any manifold.  This means that it is independent of coordinates; equivalently, the exterior derivative commutes with coordinate transformation.

The use of a reference coordinate system is a familiar concept.  Shape functions for finite elements on simplices are often defined in terms of barycentric coordinates.  Using the reference coordinate system to examine a shape function thought of as (a proxy to) a differential form is equivalent to mapping it to a reference element using pullback mapping.  

\subsection{Sobolev spaces}
Let the Sobolev semi-norms $\abs{\cdot}_{W^{k,p}(\Omega)}$ and $\abs{\cdot}_{H^k(\Omega)} = \abs{\cdot}_{W^{k,2}(\Omega)}$ have their standard meanings.  Define semi-norms and norms for any $u \in \Lambda^{(s)}(\overline{\Omega})$ as
\begin{align*}
\abs{u}_{k,\Omega} := \sum_{i \in \prox{s}} \abs{u_i}_{H^k(\Omega)}, \qquad \qquad
\norm{u}_{k,\Omega} := \sum_{r=0}^{k} \abs{u}_{r,\Omega}.
\end{align*} 
The Sobolev spaces, $H^r\Lambda^{(s)}(\Omega)$ and $\mathcal{H}^{(s),r}(\Omega)$ are then defined as the completion of  $\Lambda^{(s)}(\overline{\Omega})$ in the norms $\norm{u}_{r,\Omega}$ and $\norm{u}_{\mathcal{H}^{(s),r}(\Omega)} = \norm{u}_{r,\Omega} + \norm{du}_{r,\Omega}$ respectively.  

As a short-hand, we will write $\mathcal{H}^{(s)}(\Omega) = \mathcal{H}^{(s),0}(\Omega)$.  The spaces of proxy fields corresponding to $\cH^{(0)}(\Omega)$, $\cH^{(1)}(\Omega)$, $\cH^{(2)}(\Omega)$ and $\cH^{(3)}(\Omega)$ are the familiar $H^1(\Omega)$, $H(\curl, \Omega)$, $H(\divv,\Omega)$ and $L^2(\Omega)$.  

Note\footnote{When $r=0$, this is the observation that $H^1(\Omega)^3 \subset H(\curl, \Omega)$ and $H^1(\Omega)^3 \subset H(\divv, \Omega)$.}  that $H^{r+1}\Lambda^{(s)}(\Omega) \subseteq \cH^{(s),r}(\Omega)$ and in particular $\cH^{(0),r}(\Omega) = H^{r+1}\Lambda^{(0)}(\Omega) \cong H^{r+1}(\Omega)$, and $\cH^{(n),r}(\Omega) = H^r\Lambda^{(n)}(\Omega) \cong H^r(\Omega)$.  

For $A \in \Theta^{(s)}(\overline{\Omega})$, we similarly define 
\begin{align*}
\abs{A}_{k,\infty,\Omega} := \sum_{i, j \in \prox{s}} \abs{A^{ij}}_{W^{k,\infty}(\Omega)},\qquad \qquad
\norm{A}_{k,\infty,\Omega} := \sum_{r=0}^k \abs{A}_{r,\infty,\Omega}
\end{align*}
 and define $W^{r,\infty}\Theta^{(s)}(\Omega)$ to be the completion of $\Theta^{(s)}(\overline{\Omega})$ in $\norm{\cdot}_{r,\infty,\Omega}$.

For a given $K$ and $u \in \Lambda^{(s)}(\overline K)$ and $A \in \Theta^{(s)}(\overline K)$, define the \emph{reference semi-norms}.\footnote{These are the norms induced by the metric in which the reference coordinates are orthonormal.  They are used in the scaling argument in section \ref{sectionpyramid}.}
\begin{align*}
\abs{u}_{k,\hat K} := \sum_{\hat i \in \prox{s}} \abs{u_{\hat i} }_{H^k(\hat K)}, \qquad
\abs{A}_{k,\infty,\hat K} := \sum_{\hat i, \hat j \in \prox{s}} \abs{A^{\hat i \hat j}}_{W^{k,\infty}(\hat K)}. 
\end{align*} 
Suppose that $\left(\mathcal{T}_h\right)_{h >0}$ is a family of shape-regular partitions of $\Omega$, where every $K \in \mathcal{T}_h$ is affine equivalent to $\hat{K}$ and each $\phi_K$ satisfies
\begin{align}\label{shaperegular}
\norm{D \phi_K} \le h\quad \text{and}\quad \norm{D \phi_K^{-1}} \le \frac{\rho}{h}
\end{align}
for some $\rho \ge 1$.  For any $u \in \cH^{(s),k}(K)$ and $A \in W^{k,\infty}\Theta^{(s)}(K)$, we have the inequalities
\begin{align}
&\frac{1}{C\rho^{k+s}}\frac{h^{k+s}}{\abs{D \phi_K}^{1/2}} \abs{u}_{k,K} \le \abs{u}_{k,\hat K} \le C\frac{h^{k+s}}{ \abs{D \phi_K}^{1/2}} \abs{u}_{k, K} \label{hscalingu} \\
&\frac{1}{C\rho^{k}} h^{k-2s}\abs{A}_{k,\infty,K} \le \abs{A}_{k,\infty,\hat K} \le C\rho^{2s}h^{k-2s} \abs{A}_{k,\infty, K} \label{hscalingA}
\end{align}
for some constant $C = C(k,n)$ which is independent of $h$.  These can be deduced from the standard scaling argument for Sobolev semi-norms of functions (see, for example, \cite{ciarlet:fem}) combined with the transformation rules, \eqref{formtransform} and \eqref{tensortransform} and the observation that \eqref{shaperegular} implies that $\pd{x^{\alpha_i}}{\hat x^{\hat \alpha_j}} \le h$ and $\pd{\hat x^{\hat \alpha_j}}{x^{\alpha_i}} \le \frac{\rho}{h}$ for all $i,j$.

\subsection{Pyramidal elements}
From now on we will assume $\Omega \subset \RR^3$.  To contain the proliferation of indices, we will use the notation $(\xix,\xiy,\xiz)$ for  the reference coordinates $(x^{\hat 1}, x^{\hat 2}, x^{\hat 3})$.  The reference domain is defined as the pyramid: 
\begin{align*}
\rp = \{ (\xix,\xiy,\xiz) \:|\: 0 \le \xiz \le 1, 0 \le \xix,\xiy \le \xiz\}.
\end{align*}
We have chosen to restrict our analysis to affine maps $\phi_K$ so each element of the mesh, $K \in \mathcal{T}_h$, will be a parallelogram-based pyramid.  We will refer to a general such $K$, as an \emph{affine pyramid}.

As in \cite{phillips:pyramid} we will also use the \emph{infinite pyramid}, 
\begin{align*}
\ip = \{ (x,y,z) \: | \: 0 \le x,y \le 1, 0 \le z \le \infty \},
\end{align*}
as a tool to help analyse and define the pyramidal elements.  
The finite and infinite pyramids may be identified using the projective mapping,
\begin{align} \label{projmap}
&\phi: \ip \rightarrow \rp \\
&\phi: (x,y,z) \mapsto \left(\frac{x}{1+z},\frac{y}{1+z},\frac{z}{1+z}\right),
\end{align}
which can be thought of as a change of coordinates and so, for any element, $K$, induces the \emph{infinite pyramid coordinate system}\footnote{so-called because $z \rightarrow \infty$ as $(\xix,\xiy,\xiz) \rightarrow (0,0,1)$ at the top of the pyramid.} defined as $\tilde x = \phi^{-1} \hat x$.   We shall usually write $\tilde x = (x,y,z) $.

The corresponding weights in the change of coordinates transformation rule can be calculated explicitly:
\begin{subequations}
\begin{align}
\w{0}_\phi &= 1 \label{weight0}, \\
\w{1}_\phi &= D{\phi^{-1}}^t = (1+z)\col{1 & 0 & 0}{0& 1 & 0}{x &y &1+z} \label{jacinverse}, \\
\w{2}_\phi &= \abs{D\phi^{-1}}D\phi = (1+z)^2\col{1+z & 0 & -x}{0 &1+z & -y}{0&0&1} \label{jac}, \\
\w{3}_\phi &= \abs{D\phi^{-1}} = (1+z)^4 \label{weight3}. 
\end{align}
\end{subequations}
The approximation spaces for the finite elements presented in \cite{phillips:pyramid} are defined on the infinite pyramid using \emph{$k$-weighted} tensor product polynomials, $Q_k^{l,m,n}[x,y,z]$, which are tensor product spaces of polynomials, $Q^{l,m,n}[x,y,z]$, multiplied by a weight $\dfrac{1}{(1+z)^k}$.  That is, $Q_k^{l,m,n}$ is spanned by the set\footnote{If $l$, $m$ or $n$ is negative then $Q^{l,m,n} = \set{0}$}
\begin{align*}
\left\{\dfrac{x^ay^bz^c}{(1+z)^k},\: 0\le a \le l, 0 \le b \le m, 0 \le c \le n   \right\}.
\end{align*}  
For each family of elements on the infinite pyramid, an \emph{underlying} approximation space is defined for each order, $k \ge 1$.
\begin{itemize}
\item $H^1$-conforming element underlying space:
\begin{subequations}
\begin{align}\label{h1space}
\OUU{0}_k &= Q_k^{k,k,k-1} \oplus \spn{\frac{z^k}{(1+z)^k}}.
\end{align}
\item $\HH(\curl)$-conforming element underlying space:
\begin{align}
\begin{split}\label{hcurlspace}
\OUU{1}_k &= Q_{k+1}^{k-1,k,k-1} \times Q_{k+1}^{k,k-1,k-1} \times Q_{k+1}^{k,k,k-2} \\ 
&\oplus \left\{\frac{z^{k-1}}{(1+z)^{k+1}} \col{r_x z}{r_y z}{-r},\quad r \in Q^{k,k}[x,y]\right\}. 
\end{split}
\end{align}
\item $\HH(\divv)$-conforming element space:
\begin{align}\label{hdivspace}
\begin{split}
\OUU{2}_k &= Q_{k+2}^{k,k-1,k-2} \times Q_{k+2}^{k-1,k,k-2} \times Q_{k+2}^{k-1,k-1,k-1}  \\
&\quad \oplus \frac{z^{k-1}}{(1+z)^{k+2}}\col{0}{2s}{s_y (1+z)} \oplus \frac{z^{k-1}}{(1+z)^{k+2}}\col{2t}{0}{t_x (1+z)}, \\ 
\end{split}
\end{align}
where $s(x,y) \in Q^{k-1,k}[x,y],\quad t(x,y) \in Q^{k,k-1}[x,y].$
\item $L^2$-conforming element underlying space:
\begin{align}\label{l2space}
\UU{3}_k &= Q_{k+3}^{k-1,k-1,k-1}.
\end{align}
\end{subequations}
\end{itemize}

For an element defined on a pyramid, $K$, the underlying approximation space, $\OUU{s}_k(K)$ is defined as the space containing all the $s$-forms whose components induced by the infinite pyramid coordinate system lie in $\OUU{s}_k$:\footnote{We are using coordinate transformations here, but in \cite{phillips:pyramid}, the underlying spaces are defined as the pullbacks, $\OUU{s}_k(\hat{K}) = \set{(\phi^{-1})^*v : v \in \OUU{s}_k}$ and $\OUU{s}_k(K) = \{\phi_K^*v : v \in \OUU{s}_k(\hat{K})\}$.}
\begin{align}\label{underlying}
\OUU{s}_k(K) = \set{u \in \Lambda^{(s)}(K) : \left(u_{\tilde i}\right)_{i \in \prox{s}} \in \OUU{s}_k}.
\end{align}

By inspection, it can be seen that the exterior derivative $d: \OUU{s}_k\rightarrow \OUU{s+1}_k$ is well defined, and so, since $d$ is independent of coordinates, the exterior derivative on the spaces on each element,
\begin{align}\label{exterior}
d: \OUU{s}_k(K) \rightarrow \OUU{s+1}_k(K)
\end{align}
is also well defined. 

A full explanation of these spaces is provided in \cite{phillips:pyramid}.  Some motivation may be seen from the following lemma.
\begin{lemma}\label{UUuniformity}
For a given $K$ and $s \in \{0,1,2,3\}$ let $u \in \OUU{s}_k(K)$.  Each component  $u_{\hat \pii}$ (where $\hat \pii \in \prox{s}$) of $u$ in the reference coordinate system satisfies
\begin{align}\label{wuQkkk}
u_{\hat \pii} \circ \phi \in Q_k^{k,k,k}.
\end{align} 
This means that
\begin{align} \label{uH1}
\OUU{s}_k(K) \subset \cH^{(s)}(K).
\end{align}
\end{lemma}
\begin{proof}
The relationship between the representations of $u$ in the reference and infinite pyramid coordinate systems is given by equation \eqref{proxytransform}: $\hat u \circ \phi = \w{s}_\phi \tilde u$, where the weights, $\w{s}_\phi$, are given by \eqref{weight0}-\eqref{weight3}.  To establish  \eqref{wuQkkk}, each $s \in \set{0,1,2,3}$ needs to be dealt with as a separate case.

When $s=0$, the weight, $\w{0}_\phi = 1$ and it is clear from \eqref{h1space} that $\OUU{0}_k \subset Q^{k,k,k}_k$.  When $s=1$, inspection of \eqref{hcurlspace} reveals that $\OUU{1}_k \subset Q_{k+1}^{k-1,k,k} \times Q_{k+1}^{k,k-1,k} \times Q_{k+1}^{k,k,k-1}$.  The weight, $\w{1}_\phi = (1+z)\col{1 & 0 & 0}{0& 1 & 0}{x &y &1+z}$, so $\w{1}_\phi \tilde u \in Q_k^{k,k,k} \times Q_k^{k,k,k} \times Q_k^{k,k,k}$.  The cases $s=2$ and $s=3$ follow similarly.

Since $Q_k^{k,k,k} \subset L^\infty(\ip)$ each $u_{\hat i} \circ \phi$ is bounded on $\ip$, which means that $u_{\hat i}$ is bounded on $\hat K$ and therefore $u_i$ is bounded on $K$.  Hence $\norm{u}_{0, K}$ is finite.  By \eqref{exterior}, $du \in \OUU{s+1}_k(K)$, so $\norm{du}_{0, K}$ is finite too and $u \in \cH^{(s)}(K)$.
\end{proof}
In order to construct pyramidal elements that are compatible with neighbouring tetrahedral (and hence polynomial) elements, subspaces of the underlying approximation spaces, $\OUU{s}_k(K)$ are identified that contain only those functions whose traces on the triangular faces of the pyramid are contained in the trace space of the corresponding tetrahedral element\footnote{The trace spaces for the tetrahedral Lagrange, Nedelec edge and Nedelec face elements are given in \cite{monk:maxwell}.  By construction, the traces of the underlying spaces on the quadrilateral face of the pyramid already match those of the hexahedral elements}.  These approximation spaces are denoted $\UU{s}_k(K)$.  

Each approximation space is equipped with degrees of freedom that induce a linear interpolation operator, 
\begin{align}\label{interpoperator}
\Pi^{(s)}_{k,K} : \cH^{(s), 1/2+\epsilon}(K) \rightarrow \UU{s}_k(K),\qquad \epsilon > 0.
\end{align}
which completes the definition of the finite elements.  The necessity of the extra $1/2 + \epsilon$ regularity can be seen as a consequence of taking point evaluations at the vertices of the pyramid for the $s=0$ elements.  It is necessary for both the projection based interpolants of \cite{demkowicz:projection} and the more explicit construction given in \cite{phillips:pyramid}. 

Given a mesh, $\mathcal{T}_h$, for $\Omega$, we can assemble a global approximation space for $\cH^{(s)}(\Omega)$, 
\begin{align}\label{globalapprox}
\cV^{(s)}_h = \{ v \in \cH^{(s)}(\Omega) : v|_K \in \UU{s}_k(K) \; \forall K \in \mathcal{T}_h\}.
\end{align}
The element-wise interpolation operators respect traces on the boundary of the pyramid, i.e. 
\begin{align*}\trace u |_{\partial K} = 0 \Rightarrow \trace \Pi^{(s)}_{k,K} u |_{\partial K} = 0,
\end{align*} so we can define a bounded global interpolation operator $\Pi_h^{(s)}: \cH^{(s),1/2+\epsilon}(\Omega) \rightarrow \cV^{(s)}_h$ by $(\Pi_h^{(s)}u) |_K := \Pi^{(s)}_{k,K} (u|_K)$ for all $K \in \mathcal{T}_h$.  

\subsection{Conical product rule}
Quadrature rules on the pyramid can be deduced as special cases of the conical product rule presented by Stroud \cite{stroud:acmiac, hammermarlowstroud:simplexescones}.  Stroud defines the quadrature scheme for any continuous function, $f \in C(\hat K)$, 
\begin{align}\label{quadpyramid}
\QS(f):=\sum_{i,j,l} f(\xix_i(1-\xiz_l), \xix_j(1-\xiz_l),z_l)\lambda_i\lambda_j\mu_l.
\end{align}
He shows that given $n \ge 0$, a sufficient condition for $\QS(p) =  \int_\rp p \;d\hat x$ for any polynomial, $p \in \PP{n}(\hat x)$, is that the two one-dimensional quadrature schemes given by the points $\xix_i$ and $\xiz_l$ with respective weights $\lambda_i$ and $\mu_l$ satisfy
\begin{align} \label{quadalpha}
\sum_i \lambda_i g(\xix_i) &= \int_0^1 g(x) dx \quad \forall g \in \PP{n},\\
\sum_i \mu_i h(\xiz_i) &=  \int_0^1 (1-z)^2h(z) dz  \quad \forall h \in \PP{n}\label{quadbeta}.
\end{align}
The $k+1$ point Gauss-Legendre quadrature rule can be used to generate $\xix_i$ and $\lambda_i$ that make \eqref{quadalpha} exact for polynomials of degree $2k+1$.  The $k+1$ point Gauss-Jacobi scheme based on the Jacobi polynomial\footnote{The Jacobi polynomials, $P_{n}^{(a,b)}(s)$, $n \ge 0$, are typically defined on the interval $[-1,1]$.  Under the change of variables, $s = 2t-1$, they orthogonal with respect to the weight $(1-t)^at^b$ on the interval $[0,1]$.}, $P_{k+1}^{(2,0)}$, generates $\xiz_i$ and $\mu_i$ that make \eqref{quadbeta} exact for polynomials of degree $2k+1$.  We denote the quadrature scheme for $\hat K$ based on \eqref{quadpyramid} that uses these points and weights as $\QS_{k,\hat K}$.  The error,
\begin{align*}
E_{k,\hat K}(f) := \QS_{k,\hat K}(f) - \int_\rp f(\hat x) d\hat x.
\end{align*}
will be zero when $f \in \PP{2k+1}$.

When $f \in C(K)$, where $K$ is a pyramid equipped with a change of coordinates $\phi_K : \rp \rightarrow K$, we can define the quadrature and error functionals:   
\begin{align} 
\QS_{k,K}(f) &:= \QS_{k,\hat K}\left(\abs{D\phi_K}\hat f\right) \sim \int_K f(x) dx \label{Kkquad},\\
E_{k,K}(f) &:= E_{k,\hat K}\left(\abs{D\phi_K} \hat f\right) = \QS_{k,K}(f) - \int_K f(x) dx, \label{Etransform}
\end{align}
where $\hat f = f \circ \phi_K$, i.e. the expression of $f$ in the reference coordinate system, $\hat x$.  

\section{Pyramidal approximation spaces and quadrature}\label{sectionstroud}
If $u$ and $v$ are polynomials of degree $k$ then their product, $uv \in \PP{2k}$ so from Stroud's work, we know that $\QS_{k,\rp}(uv) = \int_\rp uv$.
In this section, we shall prove the stronger result:
\begin{theorem}\label{exacttheorem}
Let $K$ be an affine pyramid; fix $k\ge 1$; let $s \in \{0,1,2,3\}$ and let $A \in \Theta^{s}(K)$ be a constant tensor field.  Then for for any $u,v \in \UU{s}_k(K)$, the quadrature scheme $\QS_{k,K}$ exactly evaluates the product, $(u,v)_{A,K}$, i.e.
\begin{align*}
\QS_{k,K}(A^{ij}u_iv_j) =  (u,v)_{A, K}.
\end{align*}
\end{theorem}
To do this, we first need to understand exactly which functions our quadrature scheme integrates exactly.
\begin{lemma} \label{exactbaselemma}  \label{exactcor}
Suppose that $f$ is a function defined on a pyramid, $K$, and that the representation of $f$ in the infinite pyramid coordinate system: $\tilde{f} = f \circ \phi_K^{-1} \circ \phi^{-1}$, lies in the space $Q_{2k+1}^{2k+1,2k+1,2k+1}$.  Then the quadrature scheme, $\QS_{k,K}$ is exact for $f$:
\begin{align*}
\QS_{k,K}(f) = \int_{K} f d x
\end{align*}
\end{lemma}
\begin{proof}
It suffices to consider functions $p$ with a representation in the infinite pyramid coordinate system:
\begin{align*}
\tilde{p}(x,y,z) = \frac{x^ay^b}{(1+z)^{c}} \qquad 0\le a,b,c\le 2k+1,
\end{align*}
since these monomials span the space $Q_{2k+1}^{2k+1,2k+1,2k+1}$.  In finite reference coordinates, $p$ has the form $\hat p(\xix,\xiy,\xiz) = \xix^a\xiy^b(1-\xiz)^{c-a-b}$, and so, using \eqref{quadpyramid}:
\begin{align*}
\QS_{k,K}(p) &= \QS_{k}(\abs{D \phi_K} \hat p) \\
&=\abs{D \phi_K}\sum_{i,j,l} \xix_i^a(1-\xiz_l)^a\xix_j^b(1-\xiz_l)^b(1-\xiz_l)^{c-a-b} \lambda_i\lambda_j\mu_l\\
&=\abs{D \phi_K}\sum_i \lambda_i\xix_i^a \sum_j \lambda_j\xix_i^b \sum_l \mu_l(1-\xiz_l)^{c} \\
&=\abs{D \phi_K}\int_0^1 s^a ds \int_0^1 t^b dt \int_0^1 (1-\xiz)^{c+2} d\xiz. 
\end{align*}
The last step is justified because each of the sums is a quadrature rule applied to a polynomial of degree $\le 2k+1$ and so we can apply \eqref{quadalpha} and \eqref{quadbeta}.  Apply the change of variables $\xix = (1-\xiz)s$ and $\xiy = (1-\xiz)t$ to obtain:
\begin{align*}
\QS_{k,K}(p) &=\abs{D \phi_K}\int_0^1\int _0^{1-\xiz}\int _0^{1-\xiz} (1-\xiz)^{c-a-b}  \xix^a   \xiy^b \:d \xix d \xiy d\xiz \\
&=\abs{D \phi_K}\int_{\hat K} \hat p(\xix,\xiy,\xiz) d \hat x \\
&= \int_K p dx. \qedhere
\end{align*}
\end{proof}
We can now prove Theorem \ref{exacttheorem} where, in fact, we will only need Lemma \ref{exactbaselemma} to be true for $\tilde f \in Q_{2k}^{2k,2k,2k}$, which is a subspace of $Q_{2k+1}^{2k+1,2k+1,2k+1}$.

\begin{proof}[Proof of Theorem \ref{exacttheorem}]
Let $u,v \in \UU{s}_k(K)$.  $A \in \Theta^{(s)}(K)$ is a constant and so, by the first part of Lemma \ref{UUuniformity}, in infinite reference coordinates, the function $\langle u,v \rangle_A$ satisfies:
\begin{align*}
A^{\tilde\pii \tilde\pij} u_{\tilde\pij} v_{\tilde\pii} \in Q^{2k,2k,2k}_{2k}.
\end{align*}
Hence by Lemma \ref{exactcor},
\begin{align*}
\QS_{k,K} \left(A^{\pii\pij}u_\pij v_\pii\right) = \int_K A^{\pii\pij}u_\pij v_\pii = (u,v)_{A,K}.  \end{align*}
\end{proof}

Observe that for the spaces $\UU{3}_k(K)$, the integrand, $A^{\tilde\pii \tilde\pij} u_{\tilde\pij} v_{\tilde\pii} \in Q_{2k-2}^{2k-2,2k-2,2k-2}$, so we could in fact use the scheme $\QS_{k-1}$.

\section{Numerical integration and convergence}\label{sectionexisting}
Let $a:\mathcal{H}^{(s)}(\Omega) \times \mathcal{H}^{(s)}(\Omega) \rightarrow \RR$ be an elliptic bilinear form and let $V \subset \cH^{(s)}(\Omega)$ be chosen so that the problem of finding $u \in V$ such that 
\begin{align} \label{ellipticproblem}
a(u,v) = f(v) \quad \forall v \in V 
\end{align}
has a unique solution for any linear functional, $f \in V'$.  A discrete version of this problem is to find $u_h \in V_h$ such that 
\begin{align} \label{discreteellipticproblem}
a_h(u_h,v) = f(v) \quad \forall v \in V_h,
\end{align}
where $V_h$ is an approximating subspace of $V$ and $a_h$ approximates $a$ using numerical integration\footnote{We choose not to consider the effect of approximating $f(\cdot)$ by some $f_h(\cdot)$ using numerical integration because it is no different on the pyramid than for other elements.  Error estimates may be obtained by applying the standard argument and using Theorem \ref{exacttheorem}.} .  When $V_h$ is assembled using polynomial elements of degree $k$, the analysis of the effect of the numerical integration is classical; good expositions may be found in \cite{ciarlet:fem, brenner:mtfem}.  

For an example, take an elliptic bilinear form $a:H^{1}_0(\Omega) \times H^{1}_0(\Omega) \rightarrow \RR$, defined as 
\begin{align} 
a(u,v) = \int_\Omega A(d u, d v)
\end{align}
where $A \in W^{k,\infty}\Theta^{(1)}(\Omega)$ and is uniformly positive definite.   Assume that $V_h \subset H^1_0(\Omega)$ is some approximation space assembled using $k$th order polynomial  finite elements and that there exists a numerical integration rule, $S_{h,k,\Omega}(\cdot)$ which satisfies $S_{h,k,\Omega}( \partial_i u\partial_j v) = \int_\Omega( \partial_i u\partial_j v)$ for any $i$ and $j$ and all pairs of functions  $u,v \in V_h$.  Let $a_h(u,v) = S_{h,k,\Omega}(A(du,dv))$.  It is shown in \cite[page 179]{ciarlet:fem} that the solution of \eqref{discreteellipticproblem} will satisfy the error estimate:
\begin{align*}
\norm{u-u_h}_1 \le C h^k(\abs{u}_{k+1} + \norm{A}_{k,\infty}\norm{u}_{k+1}).
\end{align*}
This result is contingent on an estimate of the consistency error:
\begin{align}\label{consistency}
\sup_{w_h \in V_h} \frac{\abs{a(\Pi_h u, w_h) - a_h(\Pi_h u, w_h)}}{\norm{w_h}_1} \le Ch^k\norm{A}_{k,\infty}\norm{u}_{k+1},
\end{align}
where $\Pi_h : H^1_0(\Omega) \rightarrow V_h$ is an interpolation operator.  The constant $C = C(\Omega, k)$ is independent of $h$.  

More generally, an analysis for mixed problems can be found in \cite{brezzi:mhfem}.  The conclusion is the same:  in order to preserve an $O(h^k)$ approximation error, each bilinear form must satisfy an $O(h^k)$ consistency error estimate.

The key ingredient in the proof of the consistency error estimate, \eqref{consistency} is a local estimate:
\begin{theorem}[See \cite{ciarlet:fem}, Theorem 4.1.2] \label{ciarlet}
Given a simplex, $K \in \mathcal{T}_h$, assume that for any polynomial $\psi \in \PP{2k-2}(K)$, the quadrature error, $E_K(\psi) = 0$.  Then there exists a constant $C$ independent of $K$ and $h$ such that
\begin{align*}
&\forall A \in W^{k,\infty}(K),\quad \forall p,q \in \PP{k}(K)\\
&\abs{E_K(A(dp,dq))} \le C h^k \norm{A}_{k,\infty,K} \norm{d p }_{k-1,K} \abs{d q}_{0,K} 
\end{align*}
\qed
\end{theorem}
This theorem is proved by combining a scaling argument with the following famous result from \cite{bramblehilbert}.
\begin{theorem}[Bramble-Hilbert lemma]\label{bhlemma}
Let $\Omega \subset \RR^n$ be open with Lipschitz-continuous boundary.  For some integer $k \ge 0$ and $p \in [0,\infty]$ let the linear functional, $f : W^{k+1,p}(\Omega) \rightarrow \RR$ have the property that $\forall \psi \in \PP{k}(\Omega)$, $f(\psi) = 0$.  Then there exists a constant $C(\Omega)$ such that
\begin{align*}
\forall v \in W^{k+1,p}(\Omega), \quad \abs{f(v)} \le C(\Omega) \norm{f}_{W^{k+1,p}(\Omega)'} \abs{v}_{k+1,p,\Omega}
\end{align*}
where $\norm{\cdot}_{W^{k+1,p}(\Omega)'}$ is the operator norm.
\qed
\end{theorem}
In our more general framework, an analogous statement to Theorem \ref{ciarlet} might be:
\begin{proposition}
Let $K \in \mathcal{T}_h$ be a pyramid.  Let $s \in \{0,1,2,3\}$ and  $A \in W^{k,\infty}\Theta^{(s)}(K)$. Then
\begin{align} 
&\forall v,w \in \UU{s}_k(K) \\
&\abs{E_{K,k}(A(v,w))} \le C h^k \norm{A}_{k,\infty,K} \norm{v}_{k-1,K} \norm{w}_{0,K} \label{badlocalestimate}
\end{align}
\end{proposition}

The problem is that, unlike the situation for purely polynomial spaces, we cannot differentiate basis functions arbitrarily.  We do not have the inclusion, $\UU{s}_k(K) \subset \cH^{(s),k-1}(K)$ for $k \ge 3$.  As an example, take the $\UU{0}_k(\hat K)$ shape function associated with the base vertex, $(1,1,0)$:
\begin{align} \label{basevertex}
v(\xix, \xiy, \xiz) = \frac{\xix\xiy}{1-\xiz}.
\end{align} 
The $L^2$ norm of its third partial $\xiz$-derivative,
\begin{align}
\int_{\hat K} \left(\pd{^3v}{\xiz^3}\right)^2 d\hat x &=  \int_0^1\int_0^{1-\xiz}\int_0^{1-\xiz} \left(\frac{-6\xix\xiy}{(1-\xiz)^4}\right)^2 d\xix d\xiy d\xiz \\
&=\int_0^1 \frac{9}{(1-\xiz)^2} d\xiz,
\end{align}
is infinite.  

This means that a direct application of the argument in \cite[section 4.1]{ciarlet:fem} would fail when we attempt to use the Bramble-Hilbert lemma (Theorem \ref{bhlemma})  to obtain the estimate
\begin{align}\label{seminormestimate}
\abs{\Pi^{(s)}_{k,\hat{K}}u}_{r,\hat{K}} \le C \abs{u}_{r,\hat{K}} \qquad \forall r \in \{0, \ldots, k\}.
\end{align}

An attempt is made to avoid this problem in \cite{bergotcohendurufle::pyramid} by using the additional projector $\pi_r : H^{r+1}(K) \rightarrow \PP{r}$ satisfying 
\begin{align*}
\forall p \in \PP{r}(K) \quad \pi_r p = p
\end{align*}
on each element, $K$.  This allows element-wise estimates to be established.  Unfortunately, there is no conforming interpolant onto element-wise polynomials for pyramidal elements (see \cite{phillips:pyramid} or \cite{wieners:conforming}).  In particular, there will be discontinuities at the element boundaries, which means that $\norm{u - \pi_r u}_{1,\Omega}$ cannot be bounded.  The alternative interpretation of $\pi_r$ as a global projection onto polynomials would not allow the element-wise estimates to be obtained.

Our solution starts with the observation that not all of the members of each $\UU{s}_k(K)$ behave as badly as the function $v$ defined in \eqref{basevertex}.  Some are polynomial and others are rational but can be differentiated more times before blowing up.  For example, we will see in the proof of Lemma \ref{norms} that $v(\xi,\eta,\zeta) \xi^r \in H^{r+2}(\hat K)$.  So, we start by developing an analogue of \eqref{seminormestimate} that, formally speaking, allows us to retain as much regularity as possible. 
\begin{theorem}\label{approxdecomp}
Let $\Omega \subset \mathbb{R}^n$ be an open set with Lipschitz boundary.  Fix $\alpha \ge 0$ and let $k\ge \alpha$ be an integer.  Suppose that:
\begin{itemize}
\item $R^k \subset H^\alpha(\Omega)$ is a finite dimensional space which includes all polynomials of degree $k$;
\item $\Pi:H^\alpha(\Omega) \rightarrow R^k$ is a bounded linear projection;  
\item There exist $V_r \subset H^{r}(\Omega)$ for each $r \in \{0,\ldots,k\}$ such that we can decompose
\begin{align*}
R^k =  V_{0} \oplus \cdots \oplus V_k.
\end{align*}
\end{itemize}
Meaning that for a given $u \in H^{k}(\Omega)$, the interpolant, $\Pi u \in R^k$, may be decomposed into unique functions, $v_{r} \in V_{r}$,
\begin{align*}
\Pi u = v_{0} + \cdots + v_{k}.
\end{align*}
Then we have the following estimates for some of the functions, $v_r$:  
\begin{itemize}
\item For each $r$ satisfying $\alpha \le r \le k$:
\begin{align}\label{estimate1}
\abs{v_r}_{r} \le C\abs{u}_r.
\end{align}
\item If, additionally, $\tPP{r} \subset V_r$, where the space $\tPP{r}$ consists of polynomials of homogeneous degree, $r$, then for each $r$ satisfying $\alpha \le r+1 \le k$:
\begin{align}\label{rplus1}
\abs{v_r}_{r} \le C\abs{u}_{r+1} + \abs{u}_r.
\end{align}
\end{itemize}
\end{theorem}
\begin{proof}
For a given $r \ge \alpha$, write $W_r = V_r \cup \PP{r-1}$.  $W_r \subset R^k$ so we can let $\Psi_r:R^k \rightarrow W_r$ be any surjective linear projection.  $\Psi_r$ is a linear map between finite spaces, so the operator $(I - \Psi_r \circ \Pi): H^{r}(\Omega) \rightarrow W_r \subset H^{r}(\Omega)$ is bounded.  Also, since both $\Psi_r$ and $\Pi$ are projections, and $\PP{r-1} \subset \PP{k} \subset R^k$ we see that $\PP{r-1} \subset \ker (I - \Psi_r \circ \Pi)$.  The Bramble-Hilbert lemma gives
\begin{align*}
\norm{(I - \Psi_r \circ \Pi)u}_{r} \le C\abs{u}_{r}.
\end{align*}
By the definition of $W_r$, we have $(\Psi_r \circ \Pi) u = v_r + p$ for some $p \in \PP{r-1}$.  So $\abs{u - v_r - p}_{r} \le C\abs{u}_{r}$, which implies
\begin{align*}
\abs{v_r}_{r} &\le C\abs{u}_r + \abs{u}_{r} + \abs{p}_{r} =  (C+1)\abs{u}_r.
\end{align*}
The proof of \eqref{rplus1} follows a similar argument.  The operator $(I - \Psi_r \circ \Pi): H^{r+1}(\Omega) \rightarrow W_r \subset H^{r}(\Omega)$ is bounded because $r+1 \le k$.  The additional condition, $\tPP{r} \subset V_r$, means that $\PP{r} \subset W_r$ and so  $\PP{r} \subset \ker (I - \Psi_r \circ \Pi)$.
\end{proof}

\section{A new family of pyramidal approximation spaces} \label{newapproxspaces}
As identified in \cite{bergotcohendurufle::pyramid}, the space $\UU{0}_k$ is sub-optimal in that there exist smaller spaces which contain the same complete space of polynomials and which are compatible with neighbouring tetrahedral and hexahedral elements.  Here we will identify subspaces, $\PU{s}_k(K)$, of each of the original approximation spaces, $\UU{s}_k(K)$ that can be used to construct finite elements with the same approximation and compatibility properties and that still satisfy a commuting diagram property.  This would be an interesting exercise in its own right but within the context of this paper we shall see that the importance of these spaces is that they support a decomposition in the manner of Theorem \ref{approxdecomp} whose components still have enough ``room'' for us to apply a Bramble-Hilbert type argument in Lemma \ref{vrestimate}.  

We start the construction of these spaces in the infinite pyramid coordinate system using spaces of $k$-weighted polynomials, $Q_k^{[l,m]}$, which we define in terms of basis functions $\frac{x^ay^b}{(1+z)^c}$ where $a$,$b$ and $c$ are non-negative integers.
\begin{align}\label{Qtensordef}
Q_k^{[l,m]} = \spn{\frac{x^ay^b}{(1+z)^c}\;: \quad c \le k, \; a \le c+l-k, \; b \le c+m-k }.
\end{align}
These spaces can be characterised via a decomposition into spaces of exactly $r$-weighted polynomials,
\begin{align}\label{Qtensoralternative}
Q_k^{[l,m]} = \bigoplus_{r=0}^k Q_r^{r+l-k,r+m-k,0}.
\end{align}
It is also helpful to observe that $\frac{x^ay^b}{(1+z)^c} \mapsto \xi^a\eta^b(1-\zeta)^{c-a-b}$ under the coordinate transformation, $(\eta,\xi,\zeta) = \phi(x,y,z)$ given by \eqref{projmap}.  So if the representation in the infinite pyramid coordinate system of some polynomial $f(\hat x)$ is $\tilde f \in Q_k^{[l,m]}$ then $f$ is at most degree $k$ in $(\xi,\eta,\zeta)$ and at most degree $l$ and $m$ in $(\xi, \zeta)$ and $(\eta,\zeta)$ respectively.

Now define the spaces $\PU{s}_k$ as
\begin{subequations}
\begin{align}
\PU{0}_k &= Q_k^{[k,k]}, \label{rk0}\\
\PU{1}_k &= \left(Q_{k+1}^{[k-1,k]} \times Q_{k+1}^{[k,k-1]} \times \{0\} \right) \oplus \{\nabla u : u \in Q_k^{[k,k]} \}, \label{rk1} \\
\PU{2}_k &= \left( \{0\} \times  \{0\} \times Q_{k+2}^{[k-1,k-1]} \right) \oplus \left\{\nabla \times u : u \in  \left(Q_{k+1}^{[k-1,k]} \times Q_{k+1}^{[k,k-1]} \times \{0\} \right) \right\}, \label{rk2}\\
\PU{3}_k &= Q_{k+3}^{[k-1,k-1]}\label{rk3}.
\end{align}
\end{subequations}
The decomposition in the definitions means that  exterior derivatives, $d:\PU{s}_k \rightarrow \PU{s+1}_k$ are well defined.  The gradient is injective on $Q^{[k,k]} / \RR$; the curl is injective on $\left(Q_{k+1}^{[k-1,k]} \times Q_{k+1}^{[k,k-1]} \times \{0\} \right)$ and the divergence is a bijection from $\left( \{0\} \times  \{0\} \times Q_{k+2}^{[k-1,k-1]} \right)$ to  $Q_{k+3}^{[k-1,k-1]}$, so the sequence,
\begin{align*}
\xymatrix {
\RR \ar[r] & \PU{0}_k \ar[r]^{\nabla} & \PU{1}_k  \ar[r]^{\nabla \times} & \PU{2}_k  \ar[r]^{\nabla \cdot} & \PU{3}_k \ar[r] & 0
}
\end{align*}
is exact.  The following three lemmas relate these spaces to the $\UU{s}_k$ spaces.  To avoid the proofs distracting from our main argument we have postponed them to Appendix \ref{appproperties}.

\begin{lemma}\label{rsubsetu}
The spaces $\PU{s}_k$ are subspaces of the $\OUU{s}_k$:
\begin{align*}
\PU{s}_k \subseteq \OUU{s}_k \quad \forall s \in \{0,1,2,3\}. 
\end{align*}
\qed
\end{lemma}
In fact, for $k \ge 2$ the $\PU{s}_k$ are strict subsets of the  $\OUU{s}_k$.

For a given $s \in \{0,1,2,3\}$ and $k \ge 0$, we define the approximation space\footnote{c.f. the original approximation spaces, \eqref{underlying}} on a pyramid, $K$, as those differential forms whose infinite coordinate representation lie in $\PU{s}_k$: 
\begin{align*}
\PU{s}_k(K) = \left\{ u \in \Lambda^{(s)}(K) : \left(u_{\tilde \pii}\right) \in \PU{s}_k \right\}.
\end{align*}
These spaces still contain all the polynomials that were shown to be present in the $\UU{s}_k$ spaces.  Specifically:
\begin{lemma}\label{polynomials}
If $K$ is an affine (i.e. parallelogram-based) pyramid then, for $k \ge 1$,
\begin{align*}
&\PP{k} \subset \PU{0}_k(K) \\
&\left(\PP{k-1}\right)^{\binom{3}{s}} \subset \PU{s}_k(K) \qquad s \in \{1,2,3\}
\end{align*}
\qed
\end{lemma}

Just as with the original spaces, $\UU{s}_k(K)$, the new spaces are compatible with Nedelec's elements, which were first outlined in \cite{nedelec:80:mfea}.

\begin{lemma}\label{rtrace}
Let $K$ be a pyramid.  For each $s\in \set{0,1,2}$ there is a trace operator that takes elements of $\cH^{(s)}(K)$ to some distribution on the boundary, $\partial K$.  The image of $\PU{s}_k(K)$ under this operator consists of all traces of elements of $\cH^{(s)}(K)$ whose restriction to each triangular or quadrilateral face of $K$ is the trace of a corresponding $k$th order Lagrange, edge and face approximation function on a neighbouring tetrahedron or hexahedron.
\qed
\end{lemma}

Note that the original approximation spaces, $\UU{s}_k(K)$ were defined by explicitly identifying the subsets of underlying spaces, $\OUU{s}_k(K)$ which have polynomial trace spaces.  For the $\PU{s}_k(K)$, the polynomial trace property is inherent and this additional step is not required.  Lemma \ref{rtrace} is why we need rational functions in our spaces.  For example, there is no polynomial whose trace is the lowest order bubble on one triangular face and zero on all other faces.

From Lemmas \ref{rsubsetu} and \ref{rtrace} we see that:
\begin{corollary}\label{corsubspace}
The new approximation spaces are subspaces of the original approximation spaces.
\begin{align*}\PU{s}_k \subseteq \UU{s}_k.\end{align*} \qed
\end{corollary}

We can reuse the interpolation operators from the old spaces, \eqref{interpoperator},  to create interpolation operators for the new spaces.  Since the trace spaces of $\PU{s}_k$ are the same as $\UU{s}_k$, we just need to define projections  $\Xi^{(s)}_{k,K} : \UU{s}_k(K) \rightarrow \PU{s}_k(K)$ that do not change the trace data.

For $u \in \UU{s}_k(K)$, define $\Xi^{(s)}_{k,K}u$ to be the minimum of the functional $v \mapsto  \norm{d(v - u)}_0$ over the admissable set,\footnote{In the case $s=0$, $\PU{-1}_k$ should be interpreted as $\emptyset$}
\begin{align*}
 \left \{v \in \PU{s}_k{K} : u|_{\partial K} = v |_{\partial K} \text{ and } (u-v,dw) = 0 \:\;\forall w \in \PU{s-1}_k(K) \right \}.
 \end{align*}
Note that $\Xi^{(s)}_{k,K}$ is projection-based interpolation of $\UU{s}_k(K)$ onto $\PU{s}_k(K)$ and the well-posedness of this minimisation problem is established in \cite{demkowicz:projection}.  

Now define the maps $\Phi ^{(s)}_{k,K}: \cH^{(s),1/2+\epsilon}(K) \rightarrow \PU{s}_k(K)$ as 
\begin{align}\label{Phiinterp}
\Phi^{(s)}_{k,K} = \Xi^{(s)}_{k,K} \circ \Pi^{(s)}_{k,K}.
\end{align}  
Since both $ \Xi^{(s)}_{k,K}$ and $\Pi^{(s)}_{k,K}$ commute with $d$, so does $\Phi ^{(s)}_{k,K}$.  In fact, if  $\Pi^{(s)}_{k,K}$ is a projection based interpolant, then so is  $\Phi^{(s)}_{k,K}$.  

As with \eqref{globalapprox}, for a given $k$, we can assemble a global approximation space, 
\begin{align}\label{puglobalapprox}
\cS^{(s)}_h = \{ v \in \cH^{s}(\Omega) : v|K \in \PU{s}_k(K) \; \forall K \in \mathcal{T}_h\}
\end{align}
and define a global bounded interpolation operator $\Phi_h^{(s)}: \cH^{(s),1/2+\epsilon}(\Omega) \rightarrow \cS^{(s)}_h$ by $(\Phi_h^{(s)}u) |_K = \Phi^{(s)}_{k,K} (u|_K)$ for all $K \in \mathcal{T}_h$.  

We can construct a decomposition for these spaces that we can use with Theorem \ref{approxdecomp}.
\begin{definition} \label{defdecomp}
Given a pyramid, $K$ and $s\in \{0,1,2,3\}$ define, for each $r \ge 0$, the subspace of all the $s$-forms in $\PU{s}_k(K)$ whose components are exactly $r$-weighted when composed with $\phi:\ip \rightarrow \rp$.
\begin{align*}
\XU{s}{r,k}(K) = \left\{v \in \PU{s}_k(K) \;:\; v_{\hat \pii} \circ \phi \in Q^{r+1,r+1,0}_r \right\}.
\end{align*}  
\end{definition}
Note that although the domain of $v_{\hat \pii} \circ\phi$ is $\ip$, the condition is on the components in the reference coordinate system, $v_{\hat \pii}$, rather than the infinite pyramid coordinate system $v_{\tilde \pii}$.  In effect, what we are saying is that each $\XU{s}{r,k}(K)$ is spanned by $s$-forms whose components are linear combinations of functions 
\begin{align}\label{ebeta}
e(\xi,\eta,\zeta) = \xi^{a}\eta^{b}(1-\zeta)^{r - a - b}
\end{align}
where $a,b \le r+1$.  

\begin{lemma}\label{norms} For an affine pyramid, $K$ and for each $s \in \{0,1,2,3\}$ and $k \ge 1$, each of the spaces $\XU{s}{r,k}(K)$ satisfy the criterion for $V_r$ from Theorem \ref{approxdecomp}.  In fact,
\begin{align*}
\XU{s}{r,k}(K) \subset H^{r+1}\Lambda^{(s)}(K).
\end{align*}
Additionally, the semi-norm $\abs{\cdot}_{r,K}$ is actually a norm on each space $\XU{s}{r,k}(K)$.
\end{lemma}
\begin{proof}
Let $u \in \XU{s}{r,k}(K)$.  Each $u_{\hat i}$ can be written in terms of functions, $e(\xi,\eta,\zeta) = \xi^{a}\eta^{b}(1-\zeta)^{r - a - b}$.  When $a+b > r$, these will be rational functions with a singularity at $\zeta = 1$.  We need to understand their differentiability on the finite pyramid.  Let $\gamma = (\gamma_1,\gamma_2,\gamma_3)$ be a multi-index.  The partial derivative,
\begin{align*}
\pd{^\gamma e}{\hat{x}^\gamma} = C \xi^{a-\gamma_1}\eta^{b-\gamma_2}(1-\zeta)^{r - b - a - \gamma_3}
\end{align*}
where $C = C(\gamma,a,b,r)$ is a (possibly zero) constant dependent only on $\gamma$, $a$, $b$ and $r$.  Hence 
\begin{align}
\int_{\hat{K}} \left(\pd{^\gamma e}{\hat{x}^\gamma}\right)^2 &= C \int_0^1 \int_0^{1-\zeta} \int_0^{1-\zeta}  \xi^{2a-2\gamma_1}\eta^{2b-2\gamma_2}(1-\zeta)^{2r - 2b - 2a - 2\gamma_3} d\xi d\eta d\zeta\\
&=C \int_0^1 (1-\zeta)^{2(r+1-\gamma_1-\gamma_2-\gamma_3)} d\zeta
\end{align}  
This integral is finite if $r+1 - \abs{\gamma} > -1/2$, so $e \in H^{\lfloor r + 3/2-\epsilon \rfloor}(\hat{K})$.  By affine equivalence of $K$ and $\rp$, $u \in H^{\lfloor r + 3/2-\epsilon \rfloor}(K) \subset H^{r+1}(K)$.  

Finally, \eqref{ebeta} shows that each $e(\xi,\eta,\zeta)$ is either a rational function, or a polynomial of degree exactly $r$, so $\abs{e}_{r,\hat K} \neq 0$.  Hence $\abs{u}_{r,K} \neq 0$ and $\abs{\cdot}_{r,K}$ is a semi-norm on $\XU{s}{r,k}(K)$. 
\end{proof}

\begin{lemma} \label{decomp}For an affine pyramid, $K$ and for each $s \in \{0,1,2,3\}$ and $k \ge 1$, each of the spaces $\PU{s}_k(K)$ may be decomposed:
\begin{align*}
\PU{s}_k(K) = \XU{s}{0,k}(K)\oplus \cdots \oplus \XU{s}{k,k}(K)
\end{align*}
\end{lemma}
\begin{proof}The decomposition \eqref{Qtensoralternative} makes the claim look plausible.  The details are left to Appendix \ref{appproperties}.
\end{proof}

\section{The effect of numerical integration on the pyramid}\label{sectionpyramid}
We are ready to assemble all this machinery to prove a version of Theorem \ref{ciarlet} for pyramidal finite elements.  The first step is to establish an error estimate for each of the spaces in the decompositions in terms of the reference norms.  Recall that in \eqref{Kkquad} we defined $S_{k,K}(\cdot)$ as the $k$th order quadrature scheme for the pyramid, $K$, and that we call the error functional for this scheme $E_{k,K}(\cdot)$.  We will also use the pointwise representation, $A(u,v) = A^{ij}u_iv_j$ given in \eqref{tensorrepresentation}.
\begin{lemma} \label{vrestimate}
For any $s \in \set{0,1,2,3}$ and an affine pyramid, $K$, let $v \in \XU{s}{r,k}(K)$, $w \in \PU{s}_k(K)$  and $A \in W^{k+1,\infty}\Theta^{(s)}(K)$.  Then the error in the evaluation of the bilinear form, $\inner{v,w}_{A,K}$ using the scheme $\QS_{k, K}(\cdot)$ can be bounded in terms of the reference (semi-)norms
\begin{align} \label{localestimate}
&\abs{E_{k,K}(A(v,w))} \le C \abs{D\phi_K} \abs{A}_{k+1,\infty,\hat{K}} \abs{v}_{r,\hat{K}} \norm{w}_{0,\hat{K}}
\end{align}
where $C = C(k)$ is a constant that depends only on $k$.
\end{lemma}
\begin{proof}
We can transform the error functional onto the reference pyramid using \eqref{Etransform}.  
\begin{align}\label{Escale}
E_{k,K}(A(v,w)) = E_{k,K}\left(A^{\pii\pij} v_\pii w_\pij\right) = E_{k,\hat K}\left(\abs{D\phi}A^{\hat \pii \hat \pij} v_{\hat \pii} w_{\hat \pij}\right) = \abs{D\phi} E_{k,\hat K}\left(A^{\hat \pii \hat \pij} v_{\hat \pii} w_{\hat \pij}\right).
\end{align}
We are able to take $\abs{D\phi_K}$ outside the integral because $\phi_K$ is affine.  Define the linear functional $G \in W^{k-r+1,\infty}\Theta^{(s)}(\hat K)'$ as 
\begin{align}\label{gdef}
G(B) = E_{k,\hat K}\left(B^{\hat \pii \hat \pij} v_{\hat \pii} w_{\hat \pij}\right)\quad \forall B \in W^{k-r+1,\infty}\Theta^{(s)}(\hat K).
\end{align}
Since $\QS_k(\cdot)$ takes point values of its argument, 
\begin{align*}
\abs{G(B)} \le C\norm{B^{\hat \pii \hat \pij} v_{\hat \pii} w_{\hat \pij}}_{\infty, \hat K} \le C\norm{B}_{k-r +1,\infty,\rp} \norm{\hat v}_{\infty, \hat K} \norm{\hat w}_{\infty, \hat K}.
\end{align*}
Furthermore, all norms are equivalent on the finite dimensional spaces, $\XU{s}{r,k}(\hat K)$ and $\PU{s}_k(\hat K)$, and, by the last part of Lemma \ref{norms}, $\abs{\cdot}_{r,\hat K}$ is a norm for $\XU{s}{r,k}$.  So $G$ is continuous and $\norm{G} \le C\abs{\hat v}_{r,\hat K}\norm{\hat w}_{0,\hat K}$.  All of the equivalences of norms are done on the reference pyramid, so the constant, $C$ depends only on $k$ (in particular, it does not depend on $K$).  

From the definition of $\XU{s}{r,k}$, we know that each $v_{\hat \pii}\circ \phi \in Q_r^{r+1,r+1,0}$ and by Lemma \ref{UUuniformity} and Corollary \ref{corsubspace}, $w_{\hat \pij} \circ \phi \in Q_k^{k,k,k}$ for each $\hat \pij \in \prox{s}$.  Now suppose that $B$ is polynomial of degree $k-r$, i.e. each component, $B^{\hat \pii \hat \pij} \in \PP{k-r}$ for each $\hat \pii,\hat \pij \in \prox{s}$. Then $B^{\hat \pii \hat \pij} \circ \phi \in Q^{[k-r,k-r]}_{k-r}$.  We can assemble these facts to see that
\begin{align*}
\left(B^{\hat \pii \hat \pij} v_{\hat \pii} w_{\hat \pij} \right)\circ \phi = \left(B^{\hat \pii \hat \pij} \circ \phi \right)\left(v_{\hat \pii} \circ \phi \right)\left(w_{\hat \pij} \circ \phi \right)\in Q_{2k+1}^{2k+1,2k+1,2k+1}.
\end{align*}
So, by Lemma \ref{exactbaselemma}, the quadrature error, $E_{k, \hat K}\left(B^{\hat \pii \hat \pij} v_{\hat \pii} w_{\hat \pij} \right) = 0$.  Therefore, $\PP{k-r} \subset \ker G$ and we can apply Theorem \ref{bhlemma} (the Bramble-Hilbert Lemma) to obtain
\begin{align*}
\abs{G(A)} \le C\abs{A}_{k+1,\infty,\hat{K}} \abs{v}_{r,\hat{K}} \norm{w}_{0,\hat{K}} \quad \forall A \in W^{k-r+1,\infty}\Theta^{(s)}(\hat K)
\end{align*}
For some constant $C = C(k)$.  Substituting \eqref{gdef} and \eqref{Escale} gives the desired result.
\end{proof}

We can now apply a scaling argument to get an element-wise estimate on the quadrature error.  Recall that we defined the interpolation operator, $\Phi_K^{(s)}:\cH^{(s),1/2 + \epsilon}(K) \rightarrow \PU{s}_k(K)$ in \eqref{Phiinterp}.
\begin{lemma} 
Let $K$ be an affine pyramid satisfying the shape-regularity condition, \eqref{shaperegular}, for some $\rho \ge 1$.  Fix $s \in \{0,1,2,3\}$ and take $k \ge 2$.  Then
\begin{align}
&\forall u \in H^k\Lambda^{(s)}(K), w \in \PU{s}_k(K) \text{ and }A \in W^{k+1,\infty}\Theta^{(s)}(K)\\
&\abs{E_{k,K}(A(\Phi_{k,K}^{(s)}u, w))} \le \left(Ch^{k+1} + O(h^{k+2}) \right) \norm{A}_{k+1,\infty,K} \norm{u}_{k,K} \norm{w}_{0,K} \label{EAPhiestimate}
\end{align}
where $C = C(k)$ a constant dependent only on $k$.
\end{lemma}
\begin{proof}
Use the decomposition given in Lemma \ref{decomp} to write 
\begin{align*}
\Phi_{k,K}^{(s)} u = v_0 + \dots + v_k \text{ where } v_r \in \XU{s}{r,k}(K).
\end{align*}
By Lemma \ref{vrestimate}, we know that for each $r \in \{0\dots k\}$,
\begin{align}
\abs{E_{k, K}(A(v_r, w))}  \le C\abs{D\phi_K} \abs{A}_{k-r+1,\infty,\hat{K}} \abs{v_r}_{r,\hat K}\norm{w}_{0,\hat{K}}.
\end{align}
The interpolation operator is bounded on $\cH^{(s),1/2+\epsilon}(K)$ which is a subset of $H^{3/2+\epsilon}\Lambda^{(s)}(K)$ so Theorem \ref{approxdecomp} is applicable with $\alpha > 3/2$.  Pick some $\alpha \in (3/2, 2]$ so that when $r \ge 2$ we can use the first estimate, \eqref{estimate1}, to obtain:
\begin{align}
\abs{E_{k, K}(A(v_r, w))}  \le C\abs{D\phi_K} \abs{A}_{k-r+1,\infty,\hat{K}} \abs{u}_{r,\hat K}\norm{\hat{w}}_{0,\hat{K}}.
\end{align}
Now apply the inequalities \eqref{hscalingu} and \eqref{hscalingA} to the semi-norms (and norm) on the right-hand side to obtain
\begin{align*}
\abs{E_{k, K}(A(v_r, w))} &\le C\abs{D\phi_K}h^{k-r+1-2s}\rho^{2s}\abs{A}_{k-r+1,\infty,K} \frac{h^{r+s}}{\abs{D\phi_K}^{1/2}}\abs{u}_{r, K}\frac{h^{s}}{\abs{D\phi_K}^{1/2}}\norm{\hat{w}}_{0,K} \\
&=C h^{k+1}\abs{A}_{k-r+1,\infty,K} \abs{u}_{r, K}\norm{w}_{0,{K}},
\end{align*}
where the generic constant, $C$ still depends only on $k$.  

When $r=1$, we can similarly apply the second estimate from Theorem \ref{approxdecomp} given in \eqref{rplus1} to obtain:
\begin{align}
\abs{E_{k,K}(A(v_1, w))} \le Ch^{k+1}\abs{A}_{k,\infty,K}\left(\abs{u}_{1,K} + h\abs{u}_{2,K}\right) \norm{w}_{0,K}.
\end{align}
For $r=0$, note that $\norm{v_0}_{0,\hat K} \le C\norm{u}_{3/2 + \epsilon,\hat K} \le C \left(\abs{u}_{0,\hat K} + \abs{u}_{1,K} + \abs{u}_{2,\hat K}\right)$, so
\begin{align}
\abs{E_{k,K}(A(v_0, w ))} \le Ch^{k+1}\abs{A}_{k+1,\infty,K}\left(\abs{u}_{0,K} + h\abs{u}_{1,K} + h^2\abs{u}_{2,K}\right) \norm{w}_{0,K}
\end{align} 

Summing over the $v_r$, we obtain \eqref{EAPhiestimate}.
\end{proof}

Summing these errors over each element gives an estimate for the global consistency error due to the numerical integration (we shall ignore the $O(h^{k+2})$ terms).  Recall that in \eqref{puglobalapprox} we defined the global approximation space, $\cS^{(s)}_h \subset \cH^{(s)}(\Omega)$.  
\begin{theorem}\label{globalerror}
Let $s\in \set{0,1,2,3}$, $k \ge 2$ and assume that $\cS^{(s)}_h$ is constructed using a shape regular mesh, $\cT_h$ and finite elements, $\PU{s}_k(K)$ for each $K \in \cT_h$.  Let $A \in W^{k+1,\infty}\Theta^{(s)}(\Omega)$ and $u \in H^{(s),k}(\Omega)$.  Then the interpolant $\Phi^{(s)}_hu \in \cS^{(s)}_h$ satisifies
\begin{align*}
\sup_{w_h \in \cS_h^{(s)}} \frac{\abs{(\Phi^{(s)}_h u, w_h)_{A,\Omega} - (\Phi^{(s)}_h u, w_h)_{A,h,k,\Omega}}}{\norm{w_h}_0} \le h^{k+1} \norm{A}_{k+1,\infty,\Omega} \norm{u}_{k,\Omega}
\end{align*}
Where we define $(v, w)_{A,h,k,\Omega} := \sum_{K \in \cT_h} \QS_{K,k}\left(A(v, w)\right)$.
\end{theorem}
\begin{proof}
Let $w_h \in \cS_h^{(s)}$.
\begin{align*}
\abs{(\Phi^{(s)}_h u, w_h)_{A,\Omega} - (\Phi^{(s)}_h u, w_h)_{A,h,k,\Omega}} &\le C\sum_{K \in \mathcal{T}_h} {E_{k,K}(A( \Phi_{k,K}^{(s)}u, w_h))} \\
&\le Ch^{k+1}\sum_{K \in \mathcal{T}_h} \norm{A}_{k+1,\infty,K} \norm{u}_{k,K} \norm{w_h}_{0,K} \\
&\le Ch^{k+1} \norm{A}_{k+1,\infty,\Omega}\left(\sum_{K \in \mathcal{T}_h}\norm{u}_{k,K}^2\right)^{1/2} \left(\sum_{K \in \mathcal{T}_h} \norm{w_h}_{0,K}^2\right)^{1/2}\\
&\le Ch^{k+1} \norm{A}_{k+1,\infty,\Omega} \norm{u}_{k,\Omega} \norm{w_h}_{0,\Omega} 
\end{align*}
Dividing through by $\norm{w_h}_{0,\Omega}$ gives the result.  
\end{proof}

In the proof of Lemma \ref{vrestimate}, the important condition for $w$ was that $w_{\hat \pij} \circ \phi \in Q_k^{k,k,k}$.  So, by Lemma \ref{UUuniformity}, we could equally well have taken $w \in \UU{s}_k(K)$.  Furthermore, $\cS^{(s)}_h \subset \cV^{(s)}_h$ means that $\Phi^{(s)}_hu \in \cV^{(s)}_h$.  Hence we have a consistency error estimate for the global approximation spaces $\cV^{(s)}_h$ based on the original elements:
\begin{corollary}
Under the same assumptions as Theorem \ref{globalerror}, let $\cV^{(s)}_h$ be constructed using finite elements based on the approximation spaces, $\UU{s}_k(K)$.  Then the interpolant $\Phi^{(s)}_hu$ satisifies
\begin{align*}
\sup_{w_h \in \cV_h^{(s)}} \frac{\abs{(\Phi_h u, w_h)_{A,\Omega} - (\Phi_h u, w_h)_{A,h,k,\Omega}}}{\norm{w_h}_0} \le h^{k+1} \norm{A}_{k+1,\infty,\Omega} \norm{u}_{k,\Omega}.
\end{align*}
\qed
\end{corollary}

The error estimate may be applied to more general bilinear forms because of the commutativity $d\circ \Pi_h^{(s)} = \Pi_h^{(s+1)}\circ d$.  For example, the consistency error for the elliptic bilinear form, \eqref{ellipticproblem}, is
\begin{align*}
\sup_{v \in \cS_h^{(0)}} \frac{\abs{a(\Phi_h^{(0)}u,v) - a_h(\Phi_h^{(0)}u,v)}}{\norm{v}_1} &\le \sup_{v \in \cS_h^{(0)} }\frac{(d\Phi_h^{(0)}u,dv)_{A,\Omega} - (d\Phi_h^{(0)}u,dv)_{A,h,k,\Omega}}{\norm{dv}_0} \\
&\le\sup_{w \in \cS_h^{(1)} }\frac{(\Phi_h^{(1)}du,w)_{A,\Omega} - (\Phi_h^{(1)}du,w)_{A,h,k,\Omega}}{\norm{w}_0}\\
&\le Ch^{k+1} \norm{A}_{k+1,\infty,\Omega} \norm{du}_{k,\Omega} \\
&< Ch^{k+1} \norm{A}_{k+1,\infty,\Omega} \norm{u}_{k+1,\Omega}.
\end{align*}

A final note: as with the classical theory, the error estimates decay like $O(h^{k+1})$ but these are emphatically not $hp$-estimates.  The degree, $k$ enters into the constants in several places, which is to be expected from arguments that rely on the Bramble-Hilbert Lemma. 

\section{Conclusion}
The conventional finite element wisdom is that a $k$th order method requires a $k$th order quadrature scheme.  We have shown that this is still true for some high order pyramidal finite elements, but that the non-polynomial nature of pyramidal elements requires some unconventional reasoning to justify the wisdom.  

In the process, we have demonstrated new descriptions of families of high order finite elements for the de Rham complex that satisfy an exact sequence property.  We will examine these elements in more detail in future work, but a couple of notes are worth recording here.
\begin{itemize}
\item The approximation spaces for the first family in the sequence, $\PU{0}_k(K)$ are the same as Zaglmayr's elements, as described in \cite{demkowicz2007computing}, and which \cite{bergotcohendurufle::pyramid} describes as optimal with respect to their dimension and compatibility with neighbouring elements.  
\item Lemma \ref{polynomials} shows that the $\PU{s}_k(K)$ spaces contain polynomials corresponding to the tetrahedron of the first type.  Zaglmayr has constructed pyramidal elements containing polynomials corresponding to both types of tetrahedron, but only those corresponding to the second type are presented in \cite{demkowicz2007computing}.  It would, clearly, be interesting to compare our $\PU{s}_k(K)$ spaces with the construction for the first type. 
\end{itemize}

\begin{appendices}

\section{Properties of the new approximation spaces, $\PU{s}_k$}\label{appproperties}
In this appendix, we have collected proofs of various Lemmas in Section \ref{newapproxspaces}.
\begin{proof}[Proof of Lemma \ref{rsubsetu}]
The inclusions
\begin{align}\label{Qinclusions}
Q_n^{[l,m]} &\subseteq \left(Q_n^{l,m,\min\{l,m\}-1 } + Q_n^{0,0,\min\{l,m\} }\right) \subseteq Q_n^{l,m,\min\{l,m\} }. 
\end{align}
can be verified from the definition, \eqref{Qtensordef}.  By the first inclusion, $Q_k^{[k,k]} \subseteq Q_k^{k,k,k-1} + Q_k^{0,0,k}$, which gives the $s=0$ case: $\PU{0}_k \subseteq \OUU{0}_k$.  

The $s=0$ result implies $\nabla \PU{0}_k \subseteq \nabla \OUU{0}_k$.  Thus, since $\nabla \OUU{0}_k \subset \OUU{1}_k$, we have $\nabla Q_k^{[k,k]} \subset \OUU{1}_k$, which establishes the result for the second space in the decomposition for $\PU{1}_k$, given in \eqref{rk1}.  To deal with the first space in this decomposition, apply \eqref{Qinclusions} and the definition of $\OUU{1}_k$ given in \eqref{hcurlspace}, to obtain
\begin{align*}
\left(Q_{k+1}^{[k-1,k]} \times Q_{k+1}^{[k,k-1]} \times \{0\} \right) \subseteq \left( Q_{k+1}^{k-1,k,k-1} \times Q_{k+1}^{k,k-1,k-1} \times \{0\} \right) \subset \OUU{1}_k.
\end{align*}

The $s=2$ case may be established similarly. The space $\PU{2}_k$ is defined via a decomposition into two spaces, \eqref{rk2}.  The second space in this decomposition can be seen to be a subset of $\OUU{2}_k$ by taking curls of the $s=1$ result.  The first space is dealt with by applying \eqref{Qinclusions} directly to the definitions.  

Another application of \eqref{Qinclusions} gives $\PU{3}_k =  Q_{k+3}^{[k-1,k-1]} \subseteq Q_{k+3}^{k-1,k-1,k-1} = \OUU{3}_k$.
\end{proof}
\begin{proof}[Proof of Lemma \ref{polynomials}]
Since $\PP{k}$ is preserved by affine transformation, we can work in the reference coordinate system, $\hat x$.  Recall the components of the proxy representation of some $u \in \Lambda^{(s)}(K)$ in this coordinate system are denoted $u_{\hat \pii}$, where $\hat \pii \in \prox{s}$.  We will need to show that if all the components, $u_{\hat \pii} \in \PP{k}$ (or, for $s=1,2,3$,  $\PP{k-1}$) then $u \in \PU{s}_k(K)$.  This is equivalent to showing $\tilde u \in \PU{s}_k$, which we will do using the transformation rule, \eqref{proxytransform}, along with the explicit weights associated with the coordinate change $\phi : \ip \rightarrow \hat K$ given in \eqref{weight0}-\eqref{weight3}.  

We start with the case $s=0$.  Let $\hat u \in \Lambda^{(0)}(K)$ be any polynomial, $\hat u(\xix,\xiy,\xiz) = \xix^a\xiy^b(1-\xiz)^c$ where $a+b+c \le k$.  Then 
\begin{align*}
\tilde u =  \left(\w{0}_\phi\right)^{-1}\hat u \circ \phi = \frac{x^ay^b}{(1+z)^{a+b+c}} \in Q_k^{[k,k]} = \PU{0}_k.
\end{align*}
Similarly, for $s=3$, take $\hat u \in \Lambda^{(3)}(K)$ as $\hat u(\xix,\xiy,\xiz) =  \xix^a\xiy^b(1-\xiz)^c$ for $a+b+c \le k-1$.  Then 
\begin{align*}
\tilde u = \frac{x^ay^b}{(1+z)^{a+b+c+4}} \in Q_{k+3}^{[k-1,k-1]} = \PU{3}_k.
\end{align*}
The $s=1$ case involves a little more work.  Let $u \in \Lambda^{(1)}(K)$ have polynomial components, $u_{\hat i} \in \PP{k-1}$.  We can find $q \in \Lambda^{(0)}(K)$ with representation $\hat q \in \PP{k}$ such that $v = u - \nabla q$ has third component (in reference coordinates), $v_{\hat 3} = 0$.  By the result for $s=0$, $q \in \PU{0}_k(K)$,  and so (by \eqref{rk1}) $\nabla q \in \PU{1}_k(K)$.  We need to show that $v \in \PU{1}_k(\hat K)$.  Both $v_{\hat 1}$ and  $v_{\hat 2}$ are in $P^{k-1}$.  Suppose first that $v_{\hat 1} = \xix^a\xiy^b(1-\xiz)^c$ where $m:= a+b+c \le k-1$ and $v_{\hat 2} = 0$.
\begin{align*}
\tilde v = \left(\w{1}_\phi\right)^{-1} \hat v \circ \phi &= \frac{1}{(1+z)^2}\col{1+z & 0 & 0 }{0 & 1+z & 0 }{-x & -y & 1 } \col{\frac{x^ay^b}{(1+z)^{m}}}{0}{0} \\
&=\col{\frac{x^ay^b}{(1+z)^{m+ 1}} }{0}{-\frac{x^{a+1}y^b}{(1+z)^{m + 2}}}  = \frac{1}{(1+z)^{m+1}}\col{\left(1 - \frac{a+1}{m+1}\right)x^{a}y^b }{-\frac{b}{m+1}x^{a+1}y^{b-1} }{0} +  \tfrac{1}{m+1}\nabla \frac{x^{a+1}y^b}{(1+z)^{m + 1}}.
\end{align*}
Compare this last expression with the definition, \eqref{rk1}, to determine that $\tilde v \in \PU{1}_k$.  Note that when $a=m$ (which includes the case $a=k-1$), the first term vanishes, because $b=0$ and $1-\frac{a+1}{m+1} = 0.$\footnote{In other words, $(\xi^a,0,0)^t$ is an exact 1-form.}  An identical calculation establishes the same result when $v_{\hat 1} = 0$ and  $v_{\hat 2} = \xix^a\xiy^b(1-\xiz)^c$. 

For $s=2$, the change of coordinates formula for $u \in \Lambda^{(2)}(K)$ is
\begin{align}\label{s2change}
\tilde u = \left(\w{2}_\phi\right)^{-1} \hat u \circ \phi = \frac{1}{(1+z)^3}\col{1 & 0 & x }{0 & 1 & y }{0 & 0 & 1+z } \col{u_{\hat 1} }{u_{\hat 2}}{u_{\hat 3}}\circ \phi 
\end{align}
Suppose that $u_{\hat 1} = \xix^a\xiy^b(1-\xiz)^c$ with $m:=a+b+c \le k-1$.  Apply \eqref{s2change} to see that the contribution to $\tilde u$ is $\left(\frac{x^ay^b}{(1+z)^{m+3}},0,0\right)^t$.  Let $p = \frac{1}{m+2}\frac{x^ay^b}{(1+z)^{m+2}} \in Q^{[k-1,k-1]}_{k+1}$ and observe that $\frac{x^ay^b}{(1+z)^{m+3}} = -\pd{p}{z}$ and $\pd{p}{x} = \frac{a}{m+2}\frac{x^{a-1}y^b}{(1+z)^{m+2}} \in Q^{[k-1,k-1]}_{k+2}$ (the case $b=m$ implies that $a=0$ and therefore $\pd{p}{x} = 0$, so the final inequality in \eqref{Qtensordef} is not violated).  Hence 
\begin{align*}
 \left(\w{2}_\phi\right)^{-1}\col{\xix^a\xiy^b(1-\xiz)^c}{0}{0} \circ \phi = \nabla \times \col{0}{p}{0} - \col{0}{0}{\pd{p}{x}} \in \PU{2}_k
\end{align*}
Polynomials in the second component can be dealt with similarly.  When $u_{\hat 3} = \xix^a\xiy^b(1-\xiz)^c$, the contribution to $\tilde u$ is $\left(\frac{x^{a+1}y^b}{(1+z)^{m+3}},\frac{x^ay^{b+1}}{(1+z)^{m+3}},\frac{x^ay^b}{(1+z)^{m+2}}\right)^t$.  Hence
\begin{align*}
 \left(\w{2}_\phi\right)^{-1}\col{0}{0}{\xix^a\xiy^b(1-\xiz)^c} \circ \phi = \nabla \times \frac{1}{m+2}\col{-\frac{x^{a}y^{b+1}}{(1+z)^{m+2}}}{\frac{x^{a+1}y^b}{(1+z)^{m+2}}}{0} + \col{0}{0}{\left(1 - \frac{a+b+2}{m+2}\right)\frac{x^ay^b}{(1+z)^{m+2}}} \in \PU{2}_k.
\end{align*}
Note that $\frac{x^ay^b}{(1+z)^{m+2}} \in Q^{[k-1,k-1]}_k$ unless $a=m$ or $b=m$, but in these cases, $\left(1-\frac{a+b+2}{m+2}\right) = 0$.\footnote{Just as earth-shattering, this is the observation that $(0,0,\xi^a)^t$ and $(0,0,\eta^b)^t$ are exact 2-forms.}
\end{proof}
\begin{proof}[Proof of Lemma \ref{rtrace}]
This can be proved in an identical manner to Lemma [ref] in \cite{phillips:pyramid} for the original spaces, $\UU{s}_k(K)$.  We will just give a sketch here.  First we need to show that the restrictions of the traces of the $\PU{s}_k(K)$ functions to each face lie in the trace spaces of the corresponding tetrahedral or hexahedal approximation spaces.  Secondly we need to show that any valid trace can be achieved by some member of $\PU{s}_k(K)$.  

Convenient definitions of the tetrahedral and hexahedral spaces may be found in \cite{monk:maxwell}.  As an illustration, observe that members of the $\PU{0}_k$ which are non-zero on the face $y=0$ of the infinite pyramid can be expressed in terms of monomials $\frac{x^a}{(1+z)^c}$, where $a+c \le k$, which map to $\xix^a\xiz^{k-a-c}$, which will span all polynomials of degree $k$ on the face $\xiy=0$ of the finite pyramid, which is precisely the trace space of the $k$th order Lagrange tetrahedron..  

The second step is proved by demonstrating a linearly independent set of shape functions with non-zero traces that is the same size as the set of external degrees of freedom.  This can be done by showing that it is possible to achieve the lowest order bubble on each face, edge and vertex of the pyramid.  The sets of shape functions presented for the $\UU{s}_k(K)$ in \cite{phillips:pyramid} also suffice for the $\PU{s}_k(K)$.
\end{proof}
\begin{proof}[Proof of Lemma \ref{decomp}]
Each $\XU{s}{r,k}$ is a subset of $\PU{s}_k$, so
\begin{align*}
\XU{s}{0,k}(K)\oplus \cdots \oplus \XU{s}{k,k}(K) \subset \PU{s}_k(K) 
\end{align*}
For the reverse inclusion, we will deal with each $s \in \set{0,1,2,3}$, in turn.  For every $s \in \{0,1,2,3\}$, the transformation rule, \eqref{proxytransform}, gives $\hat u \circ \phi = \w{s}_\phi \tilde u$.

For 0-forms, the weight in the change of coordinates formula $\w{0}_\phi$ is equal to $1$ so any $u \in \PU{s}_k(K)$ satisfies $\hat u \circ \phi = \tilde u \in \PU{0}_k = Q_k^{[k,k]}$. The decomposition, \eqref{Qtensoralternative} gives
\begin{align*}
Q_k^{[k,k]} = Q_0^{0,0,0} \oplus \cdots \oplus Q_k^{k,k,0} 
\end{align*}
which is a subset of $ Q_0^{1,1,0} \oplus \cdots Q_0^{k+1,k+1,0}$ so $u \in \XU{0}{0,k}(K)\oplus \cdots \oplus \XU{s}{k,k}(K)$.

For the cases $s=1$ and $s=2$, we will consider a basis for $\PU{1}_k(K)$ and show that each element, $u$, of the basis is a member of $\XU{1}{r,k}(K)$ for some $r \in \{0 \dots k\}$, which amounts to showing that each $u_{\hat \pii} \circ \phi \in Q^{r+1,r+1,0}_r$.  

From the definition given in \eqref{rk2} it's natural to consider three cases for an element of a basis for $\PU{1}_k(K)$.  First suppose that $\tilde u \in \left(Q_{k+1}^{[k-1,k]} \times 0 \times 0\right)$ with $u_{\tilde 1} = \frac{x^ay^b}{(1+z)^c}$.  From the definition of $Q_{k+1}^{[k-1,k]}$ we see that $0 \le a \le c-2$ and $0 \le b \le c-1$ and so $2 \le c \le k+1$.  Then $\w{1}_\phi \tilde{u} = \left( \frac{x^ay^b}{(1+z)^{c-1}},0,\frac{x^{a+1}y^b}{(1+z)^{c-1}}\right)^t$ and so each $u_{\hat \pii} \in Q_{c-1}^{a+1,b,0} \subset Q_r^{r,r,0}$ where $r=c-1 \in \set{1\dots k}$.  The second case is when $\tilde u \in\left( 0 \times Q_{k+1}^{[k,k-1]} \times 0\right)$ and the reasoning is identical to the first.  Finally suppose that $\tilde u = \nabla p$ where $p = \frac{x^ay^b}{(1+z)^c} \in Q_k^{[k,k]}$.  When $c=0$, $p=1$ and $\nabla p = 0$.  So we can take $c \ge 1$ and see that each entry of
\begin{align*}
\w{1}_\phi \tilde{u} = \col{a\frac{x^{a-1}y^b}{(1+z)^{c-1}}}{b\frac{x^ay^{b-1}}{(1+z)^{c-1}}}{(a+b-c)\frac{x^ay^b}{(1+z)^{c-1}}}
\end{align*}
is in $Q_r^{r+1,r+1,0}$ for some $r \in \set{0 \dots k}$.

When $u \in \PU{2}_k(K)$, lets start with the case $\tilde u \in \left(0 \times 0 \times  Q_{k+2}^{[k-1,k-1]}\right)$ and write $u_{\tilde 3} = \frac{x^ay^b}{(1+z)^c}$.  Again, it is simple to check that each of the entries in the vector $\w{2}_\phi\tilde u = \left(-\frac{x^{a+1}y^b}{(1+z)^{c-2}},-\frac{x^{a}y^{b+1}}{(1+z)^{c-2}},\frac{x^ay^b}{(1+z)^{c-2}}\right)^t$ is in $Q_r^{r+1,r+1,0}$ for some $r \in \set{0 \dots k}$.  Now suppose that $\tilde u = \nabla \times \tilde v$ where $\tilde v \in \left(Q_{k+1}^{[k-1,k]} \times 0 \times 0\right)$ with $v_{\tilde 1} = \frac{x^ay^b}{(1+z)^c}$.  From the $s=1$ case, we know that $c \ge 2$ and so its straightforward to verify that each of the entries in 
\begin{align*}
\w{2}_\phi \tilde u = (1+z)^2\col{1+z & 0 & -x}{0 &1+z & -y}{0&0&1} \col{0}{\frac{-cx^ay^b}{(1+z)^{c+1}}}{\frac{bx^ay^{b-1}}{(1+z)^c}} = \col{\frac{-bx^{a+1}y^b}{(1+z)^{c-2}}}{\frac{-cx^ay^b}{(1+z)^{c-2}}+ \frac{-b x^ay^b}{(1+z)^{c-2}}}{\frac{bx^ay^{b-1}}{(1+z)^{c-2}}}
\end{align*}
are in $Q_r^{r+1,r+1,0}$ for some $r \in \set{0 \dots k}$.  The argument for $\tilde u = \nabla \times \tilde v$ with $\tilde v \in \left(0 \times Q_{k+1}^{[k,k-1]} \times 0\right)$ is the same.

Finally,  $u \in \PU{3}_k(K)$ means that $\tilde u \in Q_{k+3}^{[k-1,k-1]}$.  The weight $\w{3}_\phi = \frac{1}{(1+z)^4}$ so  $\hat u \circ \phi = \frac{1}{(1+z)^4} \tilde u \in Q_{k-1}^{[k-1,k-1]}$ and the reasoning is the same as the 0-form case.  
\end{proof}
\end{appendices}

\bibliographystyle{plain}  
\bibliography{../../latex/bibdesk}
\end{document}